\documentclass[12pt]{amsart}

\usepackage{amssymb}
\usepackage{amsmath}
\usepackage{amsthm}
\usepackage{amscd}
\usepackage{enumerate}
\usepackage{color}
\usepackage[usenames,dvipsnames,svgnames,table]{xcolor}
\usepackage[margin=1.11in]{geometry}%
\usepackage{hyperref}

\theoremstyle{definition}

\newtheorem{defn}{Definition}[section]

\theoremstyle{plain}
\newtheorem{theorem}{Theorem}[section]
\newtheorem{proposition}{Proposition}[section]

\newtheorem{conj}[defn]{Conjecture}
\newtheorem{cor}[defn]{Corollary}

\newtheorem{lem}[defn]{Lemma}

\newtheorem{prop}[defn]{Proposition}

\newtheorem{thm}[defn]{Theorem}

\theoremstyle{remark}

\newtheorem{rmk}[defn]{Remark}

\DeclareMathOperator{\supp}{supp}

\newcommand{\set}[1]{\left\{ #1 \right\}}
\newcommand{\paren}[1]{\left( #1 \right)}
\newcommand{\gen}[1]{\left\langle #1 \right\rangle}
\newcommand{\pkgen}[1]{\left\langle #1 \right\rangle_{K,p}}
\newcommand{\norm}[1]{\left\| #1 \right\|}
\newcommand{\abs}[1]{\left| #1 \right|}
\newcommand{\N}{\mathbb N}
\newcommand{\R}{\mathbb R}

\newcommand{\round}[1]{{\ooalign{\hfil\raise .10ex\hbox{\scriptsize#1}\hfil\crcr\mathhexbox20D}}}

\newcommand{\ff}{\mathcal{F}}

\newcommand{\pdxi}{\frac{\partial}{\partial x^i}}

\newcommand{\eng}{\mathcal{E}}
\newcommand{\Dom}{\mathcal{F}}
\newcommand{\dom}{\mathcal{F}}

\renewcommand{\Im}{\operatorname{Im}}

\newcommand{\cal}[1]{\ensuremath{\mathcal{#1}}}

\newcommand{\SG}{\operatorname{SG}}

\newcommand{\idea}{\cal{I}}
\newcommand{\grad}{\nabla}

\newcommand{\altform}{\Omega}

\newcommand{\coneform}{\altform_{C_1}^1}
\newcommand{\cform}{\altform_{C}^1}
\newcommand{\cocform}{\altform_1^{C}}

\newcommand{\conenorm}[1]{\left\Vert#1\right\Vert_{C^1}}

\newcommand{\spn}{\operatorname{span}}

\newcommand{\irt}{\mathcal{H}}
\newcommand{\loc}{\text{loc}}
\newcommand{\ltwoform}{\altform_{L^2}^1}
\newcommand{\cformnorm}[1]{\left\Vert#1\right\Vert_{\cform}}
\newcommand{\altnorm}[1]{\left\Vert#1\right\Vert_{\altform}}
\newcommand{\diralg}{\mathcal{C}}

\newcommand{\red}[1]{\textcolor{red}{#1}}

\title{Differential forms for  fractal subspaces and finite energy coordinates}
\author[Kelleher]{Daniel J. Kelleher}
\address{Department of Mathematics, Purdue University, West Lafayette, IN 47907-2067 USA}
\email{dkellehe@purdue.edu}
\thanks{Research supported in part by NSF grant DMS-0505622}

\begin{document}

\begin{abstract}

This paper introduces a notion of differential forms on closed, potentially fractal, subsets of the $\R^m$ by defining pointwise cotangent spaces using the restriction of $C^1$ functions to this set. Aspects of cohomology are developed: it is shown that the differential forms are a Banach algebra and it is possible to integrate these forms along rectifiable paths.
These definitions are connected to the theory of differential forms on Dirichlet spaces by considering fractals with finite energy coordinates. In this situation, the $C^1$ differential forms project onto the space of Dirichlet differential forms.
Further, it is shown that if the intrinsic metric of a Dirichlet form is a length space, then the image of any rectifiable path through a finite energy coordinate sequence is also rectifiable. The example of the harmonic Sierpinski gasket is worked out in detail.
\end{abstract}
\maketitle

\section{Introduction}

Considering a closed, potentially fractal, subset of $\R^m$, we define cotangent spaces as the quotients of $C^1$ functions on the space. We refer to these as $C^1$ differentials. While this is a standard construction when considering subvarieties with algebraic, smooth or analytic functions,  the current work makes use of functions which have lower regularity: only assuming one continuous derivative, and we take an arbitrary closed subset rather than an affine variety.
Theorem \ref{tkdef} establishes the equivalence of three definitions of these cotangent spaces.
In section \ref{sec:DiffForms}, $C^1$ differential forms are defined. The main result of this section is that the 1-dimensional differential forms are a Banach algebra when added to the scalar forms (functions).

In the classical setting, it is natural to define differential 1-forms as field which can be integrated along curves. This is a particular advantage of the present work --- Section 4 proves that it is possible to integrate $C^1$ differential forms along rectifiable curves which remain in the closed subset.  
\cite{CGI+13} discusses the possibility of defining integrals of differential forms on the Sierpinski gasket along curves. 
Another method, which considers integration on fractal curves is \cite{Har99}.

In section 5, if the subset has a suitable measure, it is shown that the direct integral of the cotangent bundles can be taken to define $L^2$ differential forms on the space which contain the $C^1$ forms. Theorem \ref{derham} uses integration along curves  to prove that the exterior derivative is a closed operator with respect to this direct integral structure on these $L^2$ forms.

Recently there have emerged many points of view and techniques aimed at understanding  differential structures on metric measure spaces. Works including \cite{CS07,IRT12} develope differential forms on fractals from energy measures on Dirichlet spaces, which we shall refer to as Dirichlet differential forms. This constrcution allows for the study of more general differential equations on fractal spaces.  For example, one can construct magnetic Schr\"odinger operators, as in \cite{HT14,HR16,magnets}, allowing for the rigorous mathematical study of physical objects from \cite{DABK83,Bel92}. For one-dimensional fractals Hodge theory is defined in \cite{HT14}, and \cite{HKT13} defines a Dirac operator and spectral triples with these fractals.  Navier-Stokes equations are also studied in \cite{HT14,HRT}.

Section 6 discusses the relationship between the $C^1$ differential forms and those defined for Dirichlet Spaces which have finite energy coordinates. One of the main result of this section is to prove that there is a closed projection from the $C^1$ differential forms defined to the Dirichlet differential forms. Subsection \ref{harmsg} gives the details of this relationship in the special case of harmonic coordinates on the Sierpinski Gasket as defined in \cite{Kaj12,Kaj13,Kig93-2,Kig08}.

The work in \cite{Gig15} constructs first and second order differential structures for metric spaces which satisfy Ricci curvatures lower bounds. The works \cite{BSSS12,SS12,Sma15}, construct abstract versions of Hodge--DeRham and Alexander--Spanier cohomologies for use on metric spaces. A large part of the motivation is to understand data sets by there global structure which is determined by these cohomologies.

%Section 6, considers a Dirichlet space which has a finite coordinate sequence of finite energy functions, and relate the geometry induced by Dirichlet form with that induced by the coordinate sequence. 

There is a strong relationship between geometry of a metric measure space and Dirichlet forms on the space. This is  discussed at length in \cite{Stu94,Sto10,HKT12}, where intrinsic metrics induced by the Dirichlet space are proven to be geodesic metrics in the sense that the distance is given by the length of the shortest path between two points. Theorem \ref{reccurves} of the current work proves that rectifiable curves on our fractal (with respect to the intrinsic metric) have rectifiable (with respect to Euclidean distance) images through our coordinates.

Areas of interest for further research would be extending these results to infinite dimensional spaces. This would allow for the study of Dirichlet forms with infinite coordinate sequences, as was considered in previous works \cite{Hin10}. Further, one could consider sub-Riemannian spaces, as considered in \cite{GL14}.

\noindent
%{\bf Acknowledgments:} The author would like to thank Alexander Teplyaev, whose input was crucial to this work. Also, the input of Michael Hinz was very helpful.

\section{Definitions of Cotangent spaces for closed subsets}

Consider $U$ be an open subset of $\R^m$, $C_0(\overline U)$ shall denote the set of continuous functions on the closure of $U$ vanishing at infinity, and $C^1(U)$ will denote the continuous functions which have continuous first-order partial derivatives. Define the following norm on $C^1_0(U) := C^1(U)\cap C_0(\overline{U})$
\[
\conenorm{u}:= \norm{u}_\infty +\sum_{i=1}^m\norm{\frac{\partial u}{\partial x^i}}_\infty
\]
where
\[
\norm{u}_\infty = \sup_{x\in U} |u(x)|
\]
and $\set{x^i}_{i=1}^m$ are the coordinates of $\R^m$. It is elementary to prove

\begin{prop}
$C^1_0(U)$ is a commutative Banach algebra with pointwise addition, multiplication, and the norm $\conenorm{.}$. If $U$ has compact closure, then $C^1_0(U)$ has multiplicative identity $1_U$.
\end{prop}

For a subset $K\subset U$ define
\[
\idea_K = \set{u\in C^1_0(U)~|~u(p) = 0~\text{for all }p\in K}
\]
If $K$ is a relatively closed subset of $U$, then $\idea_K$ is a closed ideal of $C^1_0(U)$. We shall take $\idea_p :=\idea_{\set{p}}$ for a point $p\in U$. 

\begin{prop}
For a given closed $K\subset U$, The space $C^1_0(K) := C^1_0(U)/\idea_K$, is a Banach algebra with the norm
\[
\norm{u}_K := \norm{u}_{\infty,K} + \inf_{v|_K = u|_K}\sum_{i=1}^d\norm{\partial_i v}_{\infty,K}
\]
where $\norm{v}_{\infty,K} = \sup_{p\in K}\abs{v(p)}$.
\end{prop}

\begin{proof}
$\idea_K$ is a closed ideal of $C^1_0(U)$ with respect to the norm above and this norm is the quotient norm of $C^1_0(U)/\idea_K$.
\end{proof}

\begin{rmk}
 We interpret this to mean that every function in an equivalence class of $C^1_0(K)$ takes the same values on $K$, so we tend to think of elements in $C^1_0(K)$ as restrictions of elements in $C^1_0(U)$.
\end{rmk}

For $f\in C^1(U)$ denote the classical gradient $\nabla f:U\to \R^m$ as  $\nabla f:= \paren{\frac{\partial f}{\partial x^1},\ldots,\frac{\partial f}{\partial x^m}}$ for all $f\in C^1(U)$.

\begin{prop}\label{prop:ip2}
If $f$ is a $C^1_0(U)$ with $f(p) = 0$ and $\nabla f(p) = 0$, then there is $g_1\in C(U)$ and $g_2(p)\in \idea_p$ such that $f( x) = g_1( x)g_2( x)+g_0( x)$ where $g_1(p) = g_2(p) = 0$, $g_1\in C^1(U\setminus\set{p})$, and $g_0$ is constant in a neighborhood of $p$. If $\overline U$ is compact, we can take $g_0\equiv 0$.
\end{prop}

\begin{proof}
By a partition of unity argument, it follows that $f$ is the sum of a function compactly supported in a neighborhood around $p$ and a function which is $0$ in a neighborhood of $p$. Thus we assume that $f$ is compactly supported around $p$ without losing generality. Further, we may assume that $p=0$.

Because $\lim_{x\to p} |f(x)|/|x| = 0$, we can find an increasing $C^1((0,N))\cap C([0,N))$  function $h(t)>\sup_{|{\bf x}|\leq t} |f(x)|/| x|$ and $h(0) = 0$. To construct this function we define $h(1/n) = \sup_{ x \leq 1/(n-1)}|f({ x})|/|{ x}|$ and interpolate with a $C^1$ function. With this $g_1({ x}) = (h(|{x}|))^{1/2}$ and 
\[
g_2({ x}) = \begin{cases}
f({ x})/(h(|{ x}|))^{1/2} & \text{if }{ x}\neq p \\
0 & \text{if } { x} = p,
\end{cases}
\]
which is continuous because
\[
\lim_{{ x} \to p} \frac{f({ x})}{(h(|{ x}|))^{1/2}} \leq \lim_{{ x} \to p}\paren{|{ x}|f({ x})}^{1/2} = 0 
\]
Clearly $g_1$ is in $C(U)\cap C^1(U\setminus\set{p})$, and 
\[
\lim_{{ x} \to p} \frac{|f({ x})|}{|{ x}|(h(|{ x}|))^{1/2}} \leq \lim_{{ x} \to p} \paren{\frac{|f({ x})|}{|{ x}|}}^{1/2}  = 0.
\]
so $g_2$ is in $C^1(U)$ and vanishes at $p$.
\end{proof}

We define
\[
\idea_{2,p} := \operatorname{clos}(\idea_p^2, \conenorm{.}),
\]
that is, the closure of the square of the ideal $\idea_p^2$ with respect to the norm $\conenorm{.}$.

\begin{cor}
The ideal $\idea_{2,p}$ consists  the set of functions of the form $g_1g_2+g_0$ where $g_0\in C_1$  is equivalently $0$ in a neighborhood of $p$, and $g_1(p) = g_2(p) = 0$ with $g_1\in C(U)\cap C^1(U\setminus\set{p})$and $g_2\in C^1_0(U)$.
\end{cor}

\begin{defn}
Define $T_p^*U := \idea_p/\idea_{2,p}$, and $\mathcal K_p$ the span of a compactly supported smooth function $\phi$ which is constant $1$ in a neighborhood of $p$. Then $C^1_0(U) = \mathcal K_p \oplus T_p^* U \oplus \idea_{2,p}$, and thus $T_p^*U$ is an $m$-dimensional vector space. The natural projection from $d_p: C^1(U) \to T^*_pU$ as the exterior derivative. Note that $d_p$ does not depend on our choice of $\phi$, because if $\psi$ was another such function, then $\phi-\psi\in \idea_{2,p}$.
\end{defn}

%Also, note that here we consider the direct sum as an ``internal'' in that for all $f \in C^1_0(U)$,
There is a unique decomponsition of $f$ into a sum of elements from $\mathcal K_p$, $T^*_pU$ and $\idea_{2,p}$ given by the Taylor expansion
\[
f(x) = f(p) \phi(x) + (x-p)\cdot \nabla f(p) +o(|x-p|^2).
\]

\begin{prop}
$d_p(fg) = f(p)d_p g + g(p) d_p f$ and 
\[
d_pf = \frac{\partial f}{\partial x^i} d_p x^i.
\]
\end{prop}
\begin{rmk}
This is classical, of course, but the following proof exhibits thinking which is useful herein.
\end{rmk}

\begin{proof}
Assuming, without loss of generality, $p=0$, the direct product decomposition (Taylor's theorem) implies that,  if $[f]$ and $[g]$ are the equivalence classes of $f,g\in C^1_0(U)$ mod $\idea_{2,p}$,
\[
[f( x)] = f(p)\phi(x) +  x \cdot \nabla f(p) + \idea_{2,p} \quad\text{and}\quad [g( x)] = g(p)\phi(x) + x \cdot \nabla g(p) + \idea_{2,p}
\]
where $\phi$ is the bump function mentioned above. The multiplying we discover
\[
[g( x)f( x)] = f(p)g(p)\phi(x) + \phi(x)(g(p)\paren{ x \cdot \nabla f(p)} + f(p)\paren{ x \cdot \nabla g(p)}) + \idea_{2,p}
\]
noting that $\phi^2 = \phi$ modulo $\idea_{2,p}$ and that $\phi$ is equivalently 1 in a neighborhood of $p$.
\end{proof}

Now, fixing a closed $K\subset U$, we define 
\[
C_0^1(K) : = C_0^1(U) / \idea_K
\]
Further, we define $\idea_p(K) = \idea_p/(\idea_K\cap \idea_p)$, alternatively 
\[
\idea_p(K) = \begin{cases}
\idea_p/\idea_K & \text{if }p\in K\\
\set0 & \text{if } p\notin K.
\end{cases}
\]
Note that this is a subspace of $C^1_0(K)$, and is a Banach algebra with the inherited norm. We define 
\[
\idea_{2,p}(K) := \idea_{2,p}/\idea_K\cap \idea_{2,p} \cong (\idea_{2,p}+\idea_K)/\idea_K,
\]
noting that both are Banach algebras, because $\idea_K$ and $\idea_{2,p}$ are both closed, so their intersection is also a closed ideal. This notation allows us to define $T_p^*K = \idea_p(K)/\idea_{2,p}(K)$. The homomorphism theorems for rings provides the following equivalent definitions.

\begin{thm}\label{tkdef}
Using the notation above,
\[
T^*_pK := \idea_p(K)/\idea_{2,p}(K)\cong \idea_p/(\idea_{2,p}+\idea_K) \cong T^*_pU/d_p(\idea_K).
\]
\end{thm}

Thus 
\begin{align*}
C^1_0(U)&  = \mathcal K_p \oplus T^*_pK \oplus (\idea_{2,p} + \idea_K) \\
& = \mathcal K_p \oplus T^*_pK \oplus \idea_{2,p}(K) \oplus \idea_K
\end{align*}
and thus 
\begin{align*}
C^1_0(K) = \mathcal K_p \oplus T^*_pK \oplus \idea_{2,p}(K).
\end{align*}
and thus we define the differential $d_p^K: C^1_0(K) \to T^*_pK$ by the natural projection associated with the above decomposition. If we define $\rho_p: T^*_pU\to T_p^*K$ and $\sigma: C^1_0(U)\to C^1_0(K)$ to be the natural projections, then $d_p^K\circ \sigma = \rho_p\circ d_p$, i.e. the following diagram commutes
\[
\begin{CD}
C^1_0(U) @>d_p>> T^*_pU\\
@VV\sigma V @VV\rho V\\
C^1_0(K) @> d^K_p >> T^*_pK
\end{CD}
\]

This  commutative diagram implies the following.
\begin{prop}
$d^K_p (fg) = f(p)d^K_p g + g(p) d^K_p f$. 
\end{prop}

\section{Definition of differential forms}\label{sec:DiffForms}

\begin{defn}
Assuming $T_p^*U$ has the standard norm, for  $\omega = \sum_{i=1}^m \omega_idx^i$, then $\norm{\omega}_p^2 = \sum_{i=1}^m \omega^2_i(p)$. Since $T_p^*K$ is a quotient space of $T_p^*U$, then we define the quotient norm for any equivalence class $[\omega]\in T_p^*K$ by
\[
\norm{[\omega]}_{T_p^*K} = \inf_{\eta\in[\omega]} \sqrt{\sum_{i=1}^m \eta_i^2}, \quad \text{for}\quad\eta = \sum_{i=1}^m\eta_idx^i.
\]

Similarly, assuming that $T^*_pU$ has the standard inner product $\gen{\eta_i dx^i,\omega_i dx^i}_p= \sum_{i=1}^m \eta_i\omega_i$,  we give $T^*_pK$ the standard quotient inner product: $\pkgen{[\eta],[\omega]} = \gen{\tilde\eta,\tilde\omega}$, where $\tilde\eta,\tilde\omega\in (d\idea_K)^\perp$ are the unique representative of $[\eta],[\omega]$ respectively.
\end{defn}

\begin{defn}
On the other hand if $T_pU$ is the tangent space at a point $p$, define the tangent space with respect to $K$ at a point $p$
\[
T_pK = \set{X\in T_pU~|~Xf = 0~\forall f\in \idea_K}.
\]
$T_pK$ is a subspace of $T_pU$, and we shall define $P_p:T_pU \to T_pK$ to be the orthogonal projection (with respect to the dot product). Since the definition is equivalent to being the set of $X\in T_pU$ such that $\omega X = 0$ for all $\omega \in d\idea_K$,  $T_pK$ is the dual of $T_p^*K$ and visa versa.
\end{defn}

If we take $\sharp:T_p^*U \to T_pU$ to be the standard musical operator such that $\sharp(\sum \omega_i dx^i) = \sum \omega_i\pdxi$ then we have that the following diagram commutes
\[
\begin{CD}
T^*_pU @>\sharp>> T_pU\\
@VV\rho V @VVP V\\
T^*_pK @> \sharp >> T_pK
\end{CD}
\]
i.e. $P\sharp = \sharp \rho$, and $\norm{[\omega]}_{T^*_pK} = \norm{P\sharp\omega}$.

Define the cotangent bundle of $U$ to be $T^*U = \coprod_{p\in U} T^*_pU$ and  $\cform(U)$ to be the continuous sections of $T^*U$ which have fiberwise bounded norms vanishing at infinity, i.e. maps from $U \to T^*U$ of the form $p \mapsto \sum_{i=1}^m g_i(p)d_px^i$ for $g_i\in C_0(\overline U)$, which we will denote $\sum_{i=1}^m g_idx^i$. Elements of $\cform(U)$ can also be interpreted as continuous functions from $U$ to $\R^m$. 

\begin{prop}
$\cform(U)$ is a Banach space with the norm
\[
\cformnorm{\omega} = \sup_{p\in U} \norm{\omega_p}_p=\sup_{p\in U}\sqrt{\omega_1^2+\omega_2^2+\cdots+\omega_m^2},\quad\text{for}\quad \omega = \omega_idx^i.
\]
\end{prop}

Similarly, define $TU = \coprod_{p\in U} T_pU$ to be the tangent bundle. We shall use $\sharp$ to refer to the musical isomorphism between $T^*U$ and $TU$.

We shall define the cotangent bundle of $K$ to be $T^*K= \coprod_{p\in U} T^*_pK$, and $\cform(K)$ to be the sections of $T^*K$ which are of the form $p\mapsto\rho_p\omega_p$ where $\omega\in \cform(U)$. That is we shall consider the maps $p \mapsto \sum_{i} g_i(p) d^K_p\tilde x^i$, where $g_i\in C(K)$, and $\tilde x^i$ is the equivalence class of $x^i$ in $C^1_0(K)$. Two forms $\omega$ and $\eta$ are equal, if $\omega_p-\eta_p$ are in $d \idea_p$.

Define $\rho: \cform(U) \to \cform(K)$ to be defined fibre-wise as the projection $\rho_p:T_p^*U \to T_p^*K$. Since convergence with respect to the sup norm on $\cform(U)$ implies pointwise convergence, it follows that that $\ker\rho$ is a closed subspace of $\cform(U)$. 

We can define the norm on $\cform(K)$ by
\[
\norm{\omega}_{\cform(K)} = \inf_{\eta \sim \omega} \norm{\eta}_{\coneform(U)}.
\]

\begin{prop}
$\cform(K)$ is a Banach space with the norm $\norm{.}_{\cform(K)}$.
\end{prop}

\begin{proof}
First, we claim that the this is a well defined norm on $\cform(K)$, and in particular that $\norm{\omega}_{\cform(K)}= 0$ if and only if $\omega_p\in d\idea_{p,2}$ for all $p\in K$. If $\norm{\omega}_{\cform(K)} =0$ this implies that $\omega \sim 0$ and thus $\omega_p \in d\idea_{p,2}$ for all $p\in K$.

On the other hand, let $\omega = \sum \omega_i(p) d^K_px^i$ is such that $\omega_p \in d\idea_{p,2}$. Because for $\omega_i\in C_0(K)$ it can be extended to a function in $\tilde\omega^i\in C_0(U)$, and defining $\tilde\omega = \sum \tilde\omega^i(p) dx^i$ is equivalent to $0$. Hence $\norm{\omega}_{\cform(K)}=0$.
\end{proof}

%\begin{prop}
%$\cform(K)$ is the span of elements of the form $d^Kf$ for $f\in C^1_b(K)$. Here $d^Kf$ represents the section which maps $p\to g(p)d_p^Kf$. In particular, $\cform(K)$ is defined by elements of the form $g_i d^Kx^i$ such that $g^i \in C(K)$. Note that representation by these elements is not necessarily unique.  
%\end{prop}
%
%\begin{proof}
%Because of the remark at the end of the previous section, that by choosing $\tilde f\in C^1_b(U)$ such that $\sigma \tilde f = f$. Then fiber-wise $\rho \circ d\tilde f = d^K f$, and $\cform(U)$ is span by elements of the from $d \tilde f$. 
%\end{proof}

On the other hand define 
\[
\cocform(K) = \set{PX~|~X\in \cocform(U)}
\]
where $\cocform(U)$ is the set of continuous vector fields on $U$, i.e. $X= \sum X^i\pdxi$ where $X^i\in C_0^1(\overline U)$, and $(PX )_p=P_pX_p$. We can consider $\sharp$ as a map from $\cform(K)$ to $\cocform(K)$, and this allows for the characterization
\[
\cformnorm{\omega}=  \sup_p \norm{P\sharp\omega}.
\]

\begin{prop}
With fiberwise multiplication, $C_0(K)$ acts on $\cform(K)$ by bounded operators. In particular $\cformnorm{f\omega} \leq \norm{f}_\infty\cformnorm{\omega}$. Similarly, if $f\in C^1_0(K)$, then defining $f\omega = \tilde f \omega$ to be fiberwise multiplication by any representative $\tilde f\in C^1_0(U)$ of $f$, then $ \cformnorm{f\omega}\leq \conenorm{f}\cformnorm{\omega}$.
\end{prop}

Now, if we define $\altform(U) = C^1_0(U) \oplus \cform(U)$ then we have can define a multiplication
\[
(f_1,\omega_1)(f_2,\omega_2) = (f_1f_2,f_1\omega_2+f_2\omega_1).
\]
\begin{thm}
$\altform(U)$ is a Banach algebra with the norm
\[
\altnorm{(f,\omega)}  = \conenorm{f}+\cformnorm{\omega}.
\]
\end{thm}

\begin{proof}
 $\altform(U)$ is a Banach space because $C^1_0(U)$ and $\cform(U)$ are both Banach spaces with their norms. To see the algebra condition

\begin{align*}
\altnorm{(f_1,\omega_1)(f_2,\omega_2)} & = \conenorm{f_1f_2} + \cformnorm{f_1\omega_1+f_2\omega_1}\\
                     & \leq \conenorm{f_1}\conenorm{f_2} +\conenorm{f_1}\cformnorm{\omega_1}+\conenorm{f_2}\cformnorm{\omega_1} \\ 
                     & \leq \altnorm{(f_1,\omega_1)}\altnorm{(f_2,\omega_2)}.
\end{align*}
\end{proof}

\section{Homology and Cohomology}

%\begin{defn}
%We define the tangent space $T_pK$ of $p$ in $K$ to be the dual of $T^*_pK$.
%\end{defn}
%If $T_pU$ is the classical tangent space at $p$: the space of derivations of $C^1_0(U)$ which satisfies the Leibniz rule at $p$. This is naturally seen to be the dual space of $T_p^*U$.  Since $T_p^*K$ is a quotient of $T_p^*U$ the linear algebra of dual spaces implies $T_pK$ can be identified with the subspace of the classical tangent space of $T_pU$ of elements $X$ such that $Xf=0$ for all $f\in\idea_K$, as this is the subspace of elements which induce elements in the dual of $T^*_pK$.  

It is important to note that if $\gamma = (\gamma_1,\ldots,\gamma_m):[a,b]\to U$ is a curve which is differentiable at $t_0$, then $\gamma$ induces an element of $\dot\gamma(t_0)\in T_pU$, where $\gamma(t_0)=p$ in a standard way
\[
\dot\gamma(t_0) f= \frac d{dt}(f\circ \gamma)(t_0) =\paren{\frac{d\gamma_1}{dt},\ldots,\frac{d\gamma_m}{dt}}\cdot \nabla f.
\]

\begin{prop}
Say $\gamma: [a,b]\to K$ is a continuous curve which is differentiable at time $t=t_0$ such that  $\gamma(t_0) = p\in K$, then $\dot\gamma(t_0)\in T_pK$, i.e. if $f$ and $g$ are elements of $C^1_0(K)$ such that $d_p^K f = d_p^K g$, then
\[
\frac{d}{dt}f\circ \gamma(t_0)=\frac{d}{dt}g\circ \gamma(t_0).
\]  
\end{prop}

\begin{proof}
Assuming $f-g \in \idea_{p,2}(K)$, the for any representatives from $C_0^1(U)$, $\tilde f,\tilde g$ respectively, $\tilde f-\tilde g \in \spn (\idea_{2,p}+\idea_K)$, that is $\tilde f-\tilde g = \phi+\psi$ where $\phi\in \idea_{2,p}$ and $\psi$ is constant on $K$, thus $\grad (\tilde f-\tilde g) = \grad \psi$ and
\[
\frac{d}{dt}(f-g)\circ \gamma(t)= \dot\gamma(t) (\tilde f-\tilde g) = \dot \gamma(t) \psi= \frac d{dt}\psi\circ \gamma(t_0) = 0.
\]  
\end{proof}

%We offer another way to integrate along paths in the as in \cite{CGI+13}.
This fact allows us to integrate differential forms in $T^*K$ along paths which stay in $K$.
%\begin{prop}
%For a rectifiable curve curve $\gamma: [a,b] \to K$ we define %$\int_\gamma d^Kf = f(b)-f(a)$, this definition is well defined on %$dC^1_b(K)$.
%\end{prop}
%
%\begin{proof}
%say $f$ and $g$ are such that $d^Kf = d^Kg$ (where is equality is taken fiber-wise), 
%\end{proof}
%
%
\begin{theorem}\label{derham}
For a rectifiable curve $\gamma:[a,b]\to K$, the linear functional from $\cform(K)\to \R$
\[
\omega \mapsto \int_\gamma \omega := \int_\gamma \eta 
\] 
where $\eta\in\cform(U)$ is fiber-wise a representative of $\omega$, is well defined.
\end{theorem}

\begin{proof}
First, we know that we can chose such an $\eta$ by the definition of $\cform(K)$. Say $\eta$ and $\eta^\circ$ are two representatives of $\omega$, as proven in the proposition above $\dot\gamma(t)(\eta-\eta^\circ) = 0$ for all $t$, thus
\[
\int_\gamma \eta = \int_a^b \dot\gamma(t)\eta \ dt = \int_a^b \dot\gamma(t)\eta^\circ \ dt =  \int_\gamma \eta^\circ.
\] 
\end{proof}

This implies a version of the fundamental theorem of line integrals for the restricted cotangent space.

\begin{thm}\label{ftpi}
For any rectifiable curve $\gamma:[a,b]\to K$ and any function $f\in C^1_0(K)$, $\int_\gamma \ d^K f = f(b) - f(a)$. 
\end{thm}

\begin{thm}
If for every two points in $x,y\in K$ there is a rectifiable curve $\gamma: [a,b]\to K$ such that $\gamma(a) = x$ and $\gamma(b) = y$, then $d^K: C^1_0(K) \to \Omega^1_C(K)$, is closed as an operator from $C_0(K) \to\Omega^1_C(K)$, where $C_0(K)$ is the continuous functions vanishing at infinity with the uniform norm.
\end{thm}

\begin{proof}
Say that $f_i \to f$ and $\omega_i \to \omega$, for $f_i,f\in C^1_0(K)$ and $\omega_i,\omega\in \cform(K)$ with $d^Kf_i = \omega_i$. Pick $x_0\in K$, then, for each $x\in K$ pick $\gamma_x:[a,b] \to K$ with $\gamma_x(a) = x_0$ and $ \gamma_x(b) = x$, then we can define
\[
g(x) = f(x_0)+\int_{\gamma_x} \omega. 
\]
Note that by the theorem \ref{ftpi}, this function is independent of our choice of $\gamma_x$.
It is also clear that
\[
f_i(x) = f_i(x_0) + \int_{\gamma_x}\omega_i
\] 
by the theorem \ref{ftpi}.
Finally, because $\dot\gamma(t)\omega_i  \to \dot\gamma(t)\omega$ uniformly, we have that $f=g$.
\end{proof}

\section{If $K$ is a metric measure space}

In this section, concepts from measurable spaces are introduced to help us understand the underlying space from an intrinsic viewpoint. Assume that $K$ is endowed with $\sigma$-finite Borel regular measure $\mu$.

%Associate $T_p^*K = (d_p\idea_K)^\perp$, where $\perp$ is considered with respect to the inner product from the last section. In this section we shall think of $T^*_pK$ as the subset of $T_pU$ of differentials which vanish on all function $\idea_K$ and refer to it as the tangent space of at $p$ with respect to $K$.  T

The subset $TK=\coprod_{p\in K}T_pK$ is closed in $TU$ because if $(x_n,X_n) \to (x,X)$ in $TU$ and $X_nf = 0$ for all $f\in \idea_K$, then $Xf = 0$.
Consider $\pkgen{.,.}$ to be the quotient inner product induced by the the standard metric on $T^*_pU$, and $\norm{.}_{K,p}$ be the associated norm. If $P_p:T_pU \to T_pK$ is taken to be the orthogonal projection onto the tangent space, then if $X_i = \frac{\partial}{\partial x^i}$ is the standard frame for $TU$, then $\norm{dx^i}_{K,p} = \norm{P_p X_i}$.

\begin{lem}
The values  $\pkgen{d^Kx^j,d^Kx^i}$ are measurable on $K$ and hence $T_x^* K$ is a measurable field.
\end{lem}

\begin{proof}
If $p_k\to p$, then  $\norm{d^K_px^i}_{K,p} \geq \limsup_k \norm{d^K_{p_k} x^i}_{K,p}$. This follows because, taking $T_{p} K$ to be the tangent space at $p$,  $P_{p_k}: T_{p_k}U \to T_{p_k}K$ to be the orthogonal projection, this means that $\norm{d^K_p x^i}_{K,p} = \norm{P_{p} X^i}$. 

If we restrict to a subsequence such that $\norm{d^K_{p_k} x^i}_{K,p}$ converges to something greater than 0 (the claim is trivial if such a subsequence does not exist). Then, because the sequence is eventually contained in a compact neighborhood of $(p,0) \in TU$, there is a further subsequence that $Y_k = P_{p_k}X_i$ converges to $Y\in T_pU$. Since the norm is continuous, 
\[
Y = \lim_{k\to\infty} \frac{\gen{Y_k,X_i}}{\norm{Y_k}^2} Y_k = \lim_{k\to\infty} \frac{\gen{Y,X_i}}{\norm{Y}^2} Y
\] 
but since $Y$ is in $T_p K$ this implies that 
\[
\norm{P_p X_i} \geq \frac{\gen{Y,X_i}}{\norm{Y}} = \lim_{k\to\infty }  \norm{d^K_{p_k} x^i}_{K,p},
\]
\end{proof}

Since, $T_pK$ is a measurable field over $K$, it is possible to consider the direct integral of this field, as follows.

\begin{defn}
For the rest of this section we shall assume that $\mu$ is a Radon measure on $U$ with support in $K$. We define measurable forms on $K$ by the direct integral 
\[
\ltwoform(K,\mu) = \int_{U}^\oplus T_x^*K \ d\mu(x).
\]
We shall write $\ltwoform(K)$ when the choice of measure is clear. This is also a Hilbert space with
\[
\gen{\omega,\eta}_K  = \int_K\pkgen{\omega_p,\eta_p} \ d\mu(p).
\]
\end{defn}

We consider $\cform(K)\subset \ltwoform(K)$ in the natural way.

\begin{thm}
$\cform(K)$ is a dense subspace of $\ltwoform(K)$. 
\end{thm}

\begin{proof}
Every element in $\ltwoform(K)$ is the space of measurable sections $\sum_{i=1}^m \omega_i d^K x^i$ for $\omega^i\in L^2(K,\mu)$. In this case
\[
\norm{\omega}_{\omega_{L^2}(K,\mu)} = \sum_{i,j=1}^m \int \omega_i\omega_j \gen{d^Kx^i,d^Kx^j} \ d\mu
\] 
By approximating each $\omega_i$ by continuous functions with respect to $\mu$ and noting that $|\gen{d^Kx^i,d^Kx^j}|\leq 1$, one sees that $\omega$ can be approximated by a forms with these coefficients.
\end{proof}

If $\nu$ is another Radon measure with support on $K$ with $\nu \ll \mu$, then there is a natural map $\ltwoform(K,\nu)\to \ltwoform(K,\mu)$ by $\omega \mapsto \paren{d\mu/d\nu}\omega$.

\section{Two notions of 1-forms}\label{sec:21forms}

In this section we shall compare the differential forms from \cite{IRT12,CS07,HRT} to $C^1$ differential forms. First we shall recall some basic definitions and concepts from the theory of Dirichlet forms, for more see \cite{BH91, FOT11}. For a locally compact metric space $K$, consider a regular Dirichlet form $(\eng,\Dom)$ on $L^2(K,\mu)$, where $\mu$ is a Radon measure. That is $(\eng,\Dom)$ satisfy
\begin{enumerate}
[(DF1)]
\item $\Dom\subset L^2(K)$ is a dense subspace and $\eng:\Dom\times\Dom\to \R$ is a non-negative definite symmetric bilinear form.
\item \emph{Closed}: $\Dom$ is a a Hilbert space with the inner product 
\[
\eng_1(f,g) := \eng(f,g) + \gen{f,g}_{L^1(K,\mu)}.
\]
\item \emph{Markov property}: $f\in \Dom$ implies that $\hat f = (0\vee f)\wedge 1 \in \Dom$ and  $\eng(\hat f ,\hat f) \leq \eng (f,f)$.
\item \emph{Regular}: The space $\diralg:= C_c(K)\cap \Dom$ is uniformly dense in $C_c(K)$ and dense in $\Dom$ with respect to the norm induced by $\eng_1$.
\end{enumerate} 

%Since $K$ is assumed to be as subset of $\R^m$, it follows that $K$ is locally compact.
It shall also be assumed that $(\eng,\Dom)$ is strongly local: For $u,v\in\Dom$, if $u$ is constant in the support of $v$, then $\eng(u,v) =0$.

From \cite[Chapter I 3.3]{BH91}, the space $\diralg$ from (DF4) is an algebra with respect to pointwise multiplication and addition, thus we shall refer to it as the Dirichlet algebra. For any $f\in \diralg$ define the energy measure 
\[
\int \phi d\Gamma(f) = \eng(\phi f,f) - \frac12\eng(\phi,f^2) \text{ when } \phi \in \diralg.
\]
Consider the space $\diralg\otimes\diralg$ with the bilinear form 
\[
\gen{a\otimes b,c\otimes d}_\mathcal{H} = \int_K bd \ d\Gamma(a,c).
\]
As in \cite{HRT}, define the differential forms on the space $K$ to be the space $\mathcal H$, which is attained from $\diralg\otimes\diralg$ by factoring out the zero space of $\gen{.,.}_\mathcal{H}$ and completing.

Here we shall assume, that there is a finite coordinate sequence  $\Phi = (\phi^i)_{i=1}^m$ of finite energy functions $\phi^i:K \to \R$, that is
\begin{enumerate}
[(CO1)]
\item for all $i,j\in \N$, $d\Gamma(\phi^i,\phi^j)/d\mu \in L_1(K,\mu)\cap L_\infty(K,\mu)$,
\item the space of functions 
\[
C^1(\Phi) = \set{F(\phi^{1},\ldots,\phi^{m})~|~F\in C^1_0(\R^n)}
\] 
%of the cylindrical functions
is dense in $\diralg$ with respect to $\eng_1$ (the inner product from (DF2)).
\end{enumerate} 
%We make the additional assumption on $\Phi$ that for any point $x\in K$ there a neighborhood of $x$, there is a finite an compact neighborhood $U_x$ of $K$ such that there is a finite subset $I_x\subset \N$ such that $\phi^i|_{U_x} \equiv 0$ for $i\notin I_x$ (we shall assume that $I_x$ is the minimal such set).

Notice, that the assumption that $C^1(\Phi)$ is dense in $\mathcal C$ implies that it is dense in $C_0(K)$ which means that it is point separating, so that it is an embedding into the space.

We shall refer to $\Phi$ as a function from $K\to\R^m$. We shall define $K_\Phi = \Phi(K)$ and $\tilde{\mu}=\Phi^*\mu$, where $\mu$ is a $\sigma$-finite Borel regular measure with full support. Because $\Phi$ is point separating, it is a homeomorphism between $K$ and $K_\Phi$.

The energy measures $\Gamma$ satisfy the following chain rule from theorem 3.2.2  in \cite{FOT11}: if for $f,h, g_1,\ldots,g_k \in \Dom$ and $F\in C^1(\R^k)$, if $f=  F(g_1,\ldots, g_k)$, then
\[
\Gamma(f,h) = \sum_{i=1}^k \frac{\partial F}{\partial x^i} \Gamma(g_i,h).
\]

\begin{proposition}\label{prop:Density}
Elements of the form $\sum_{i=1}^m\phi^i\otimes \omega_i$ for $\omega_i\in\diralg$ are dense in $\mathcal H$.
\end{proposition}

\begin{proof}
Using the chain rule for $\Gamma$ one can show that 
\[
F(\phi^{i_1},\ldots,\phi^{i_n})\otimes \omega = \
\sum_{k=1}^n \phi^{i_k}\otimes\paren{\omega\frac{\partial F}{\partial x^{i_k}}} 
\]
and thus  result then follows because $C^1(\Phi)$ is dense in $\dom$.
\end{proof}

Next, we define the map between differential forms on $U$ to $\irt$.

\begin{defn}
Define, for $\omega = \sum_{i=1}^m \omega_i dx^i$, the map $\pi: \cform(U) \to \irt$, by 
\[
\pi \omega = \sum_{i=1}^m \phi^i\otimes (\omega_i\circ \Phi)\]
and define the semi-norm $\norm{\omega}_Z = \norm{\pi\omega}_\irt$. Define 
\[
\mathcal N = \Omega_{C}^1 (U)/\ker\pi \cong \pi(\Omega_{C}^1(U))
\]
\end{defn}

%\begin{proof}
Elements of the form $\sum_{i=1}^m\phi^i\otimes\omega_i = \pi\Omega_{C}^1(U)$ are dense in $\irt$ from proposition \ref{prop:Density}.
%\end{proof}

\begin{prop}\label{prop:seminormZ}
The seminorm $\norm{.}_Z$ has the following formula,
\[
\norm{\omega}_Z^2 = \int_{K} \tilde\omega_i Z^{ij}\tilde\omega_j \ d\mu(x)
\]
where $\tilde\omega_i = \omega_i\circ\Phi$ and 
\[
Z^{ij}_x  = \frac{d\Gamma(\phi^i,\phi^j)}{d\mu}(x).
\]
%Thus, defining $Z = \pi^*\pi$, then $Z$ acts by fiberwise multiplication by the matrix $Z^{ij}$. . 
Thus $\norm{d f}_{Z}^2 = \eng(f,f)$ and 
$d\idea_{K_\Phi} \subset \ker\pi$. Further, if $ P_x$ is the orthogonal projection $T^*_{\Phi(x)}U \to T_{\Phi(x)}K_\Phi$, then $P_xZ_xP_x=Z_x$ for $\mu$-almost every $x$. 
\end{prop}
\begin{proof}
To see the formula, let $\omega = \sum_{i=1}^n \omega_idx^i $ and $\tilde \omega_i = \omega_i\circ\Phi$,
\[
\gen{\pi\omega,\pi\omega}_{\pi} = \sum_{i,j} \gen{\phi^i\otimes \tilde\omega_i,\phi^j\otimes \tilde\omega_j}  = \sum_{i,j} \int \omega_iZ^{ij}\omega_j \ d\mu .
\]
Thus, by the chain rule,
\[
\norm{df}_Z = \sum_{i,j} \int \frac{\partial f }{\partial x^i}\frac{\partial f}{\partial x^j} d \Gamma(\phi^i,\phi^j) = \eng(f\circ\Phi,f\circ\Phi).
\]
For any function $f\in \idea_{K_\Phi}$, $f\circ \Phi$ is a constant. Thus $\norm{\pi df}_\irt = \eng(f\circ \Phi) = 0$.

This implies that $P\omega = \sum_{i=1}^m P_x\omega_idx^i$ is the fibrewise projection from $\cform(U)$ to $\cform(K)$, then $\gen{P\eta,P\omega}_Z =\gen{\eta,\omega}_Z$ for any $\eta,\omega\in\cform(U)$. Since $\int_K \tilde\omega_i (PZP)^{ij} \tilde\eta_j \ d\mu = \int_K\tilde\omega_iZ^{ij} \tilde \eta_j \ d\mu$, it implies that $(P_xZ_xP_x)^{ij} = Z^{ij}_x$ for almost all $x$
\end{proof}

\begin{thm}
~
\begin{enumerate}
\item $\pi$ embeds $\mathcal N$ into a dense subspace of $\irt$.

\item Because $\Omega_C^1 (K_\Phi) = \Omega_C^1 (U)/d\idea_{K_\Phi}$, and $d\idea_{K_\Phi} \subset \ker\pi$, $\pi$ descends to a homomorphism $\tilde\pi: \Omega_C^1(K_\Phi) \to \irt$, which is given by the formula $\sum_i \omega_id^Kx^i \mapsto \sum_i \phi^i\otimes ( \omega_i\circ\Phi)$. 

%Assume $\mu$ is such that $\mu \ll \Gamma(\phi_i)$ for all $i = 1,2,\ldots,m$. 
\item As a function from $\ltwoform(K) \to \irt$,  $\pi$ is a densely defined closable operator and $\pi^*$ is given by
\begin{align}\label{zformula}
\pi^*\sum_{i=1}^m\phi^i\otimes \omega_i = \sum_{i,j}^m Z^{ij}_*\omega_i d^Kx^j.
\end{align}
where 
$
Z_*^{ij}(x)  =  Z^{ij}_{y}$ if $x = \Phi(y) $.
\end{enumerate}
\end{thm}

\begin{proof}
Part (1) is a result of the fact that the image of $\pi$ are tensors of the form $\sum_{i=1}^n \phi^i \otimes \omega_i$ for $\omega_i\in C(K)$, which are dense in $\irt$ by proposition \ref{prop:Density}. (2) is a corollary of proposition \ref{prop:seminormZ} and the homomorphism theorems.

(3) Because $\cform(U)$ is dense in $\ltwoform(U)$, $\pi$ is a densely defined operator there.  Because elements of the form $\sum_{i=1}^m \phi^i \otimes \omega_i$, $\omega_i\in \diralg$ are dense in $\mathcal H$, formula (\ref{zformula}) implies $\pi^*$ is densely defined, and hence $\pi$ is closable.

To see the formula,
\begin{align*}
\gen{\pi \sum_{i=1}^m\eta_idx^i,\sum_{i=1}^m \phi^i\otimes \omega_i} 
 &  = \sum_{i,j=1}^m\int \tilde\eta_jZ^{ij}\omega_i d\mu(x) \\
 & = \sum_{i,j=1}^m\int \tilde\eta_j(PZP)^{ij}\omega_i d\mu(x) \\
& = \gen{\sum_{i=1}^m\eta_idx^i,\pi^*\sum_{i=1}^m\phi^i\otimes \omega_i}.
\end{align*}
\end{proof}

We define $Z:\ltwoform(K)\to \ltwoform(K)$ by $Z:= \pi^*\pi$

\begin{lem}\label{gammaandz}
For all $f,g\in C^1_0(K)$, then 
\[
\pkgen{d^K_pg,(Z d^Kf)_p} = \frac{d\Gamma(f\circ\Phi,g\circ \Phi)}{d\mu}(p) = \sum_{i,j=1}^m Z^{ij}_p \frac{\partial f}{\partial x^i}\frac{\partial g}{\partial x^j}
\]
for almost all $p$. In particular, since 
\[
\pkgen{d^K_pf,(Zd^Kf)_p}=\frac{d\Gamma(f)}{d\mu}(p)\geq 0\quad \text{a.e.},
\]
$Z$ acts on almost every fiber of $\ltwoform(K)$ by the matrix $Z_p^{ij}$ which is non-negative definite.
\end{lem}

%Thus, if we consider $Z_p: T_pU\to T_pU$, the image of $Z_p$ is contained in $T_pK$. In particular $Z_p$ is a well defined map from $T_pK \to T_pK$. Thus we can define $Z^*_p: T_p^*K\to T_p^*K$ by $Z^*_p\omega (X) = \omega (Z_pX)$, for all $\omega \in T_p^*K$ and $X\in T_pK$. $Z$ is a pointwise a symmetric matrix and is non-negative definite. The fact that it is non-negative definite is because of the above corollary, and that $\Gamma(f)$ is a positive measure for all functions $f$.

Give $\mathcal{H}$ a $\mathcal{C}$-module structure on simple tensors with right action for $a,b,c\in\mathcal C$
\[
a(b\otimes c) : =  b\otimes (ca)
\]

\begin{rmk}
This differs slightly from the bimodule structure which is typically given to $\irt$. There is no loss in generality here because we are assuming that $\eng$ is local, and hence the chain rule for energy measures implies the left and the right actions coincide. 
\end{rmk}

 $\coneform(K)$ is a  $C(K)$-module with fiber-wise scalar multiplication: for $\phi\in C(K)$ and $\omega\in \coneform(K)$,
$
(\phi\cdot\omega)_p = \phi(p)\omega_p.
$
Then $\pi$ is a $\mathcal C$-module homomorphism, that is $\pi(\phi\cdot \omega) = \phi \cdot\pi\omega$ for all $\omega\in\coneform(K)$, and $\phi\in\mathcal C$.

Because $\Gamma(\phi^i)\leq c_i\mu$ for some constant $c_i$, we have that there is a constant $c$ such that  $c\norm{\omega}_{\ltwoform(K)} \geq \norm{\pi\omega}_\mathcal{H}$.

%Notice that $\pi^*\pi: \ltwoform(K_\Phi)\to\ltwoform(K_\Phi)$ is defined by $\sum_{i=1}^m \omega_{i}d^{K_\Phi} x^i \mapsto \sum_{i,j=1}^m Z_{ij}\omega_i d^{K_\Phi}x^j$.

% $\norm{Z df}_{\ltwoform(K_\Phi)}^2 = \eng(f)$ so that $Z df = 0$ ($d$ is the standard exterior derivative here), only if $f$ is locally constant of $K_\Phi$. Thus $\eng(f) = 0$ implies that $d^{K_\Phi}f = 0$ for all $f\in C^1(\R^m)$. This implies that $Z: \Omega^1_L\to\Omega^1_L$ is a well defined map.

Consider the metric 
\[
\rho_\mu(x,y) = \sup\set{f(x)-f(y)~|~\norm{d\Gamma(f)/d\mu}_\infty \leq 1}
\]
on the space $K$. This is called the intrinsic metric of on $K$ with respect to the Dirichlet form as in, for example, \cite{Stu94}.
\begin{prop}
If $\Phi$ is a coordinate sequence as above, and $Z$ is the related matrix as above
\[
\rho_\mu(x,y) = \sup\set{F\circ\Phi(x)-F\circ\Phi(y)~|~F\in C^1_0(U),~ \norm{\pkgen{d^K F,Z d^K F}}_\infty\leq 1}.
\]
\end{prop}

In the rest of this section it is assumed that $\rho_\mu$ induces the original topology on $K$. This assumption and the fact that $\Phi$ is a coordinate sequence for $\mu$, implies, by \cite{HKT13}, $\rho_\mu$ forms a shortest path (geodesic) distance on $K$. From corollary \ref{gammaandz}, we get that the following definition is equivalent to the above. That is if $\gamma: [a,b]\to K$, we define its length
\[
L_\mu(\gamma) = \sup\set{\sum_{i=1}^k \rho_\mu(\gamma(t_{i-1}),\gamma(t_{i}))~|~a= t_0< t_1< \cdots < t_m = b},
\]
then 
$
\rho_\mu(x,y) = \inf\set{L_\mu(\gamma)~|~\gamma:[a,b]\to K,~\gamma(a) = x,~\gamma(b) = y }$.

If we define $\lambda_m(Z_x)$ to be the largest eigenvalue of $Z_x$, then we have the following lemma. 
\begin{lem}\label{Zopnorm}
There is a constant $c_Z$ depending only on  $\Phi$ such that
\[
\pkgen{d^K_p F,(Z d^KF)_p} \leq  c_Z\norm{d^K_pF}^2_{T^*_pK}\leq c_Z\abs{\nabla F(p)}
\] 
almost everywhere, where $c_Z\leq \norm{\lambda_m(Z_x)}_{\infty}\leq\norm{\operatorname{Tr} Z_x}_\infty$. %Thus $\rho_\mu(x,y) \geq |\Phi(x)-\Phi(y)|$.
\end{lem}

\begin{theorem}\label{reccurves}
If $\gamma:[a,b]\to K$ is a curve such that $L_\mu(\gamma)< \infty$, then $\tilde\gamma:=\Phi\circ \gamma$ is a rectifiable curve in $\R^m$, and the Euclidean length of $\gamma$ is bounded by $c_ZL_\mu(\gamma)$.
\end{theorem}

\begin{proof}
Because of Lemma \ref{Zopnorm}, assuming that $\gamma$ is a unit-speed parametrization (i.e. $\rho_\mu(\gamma(t),\gamma(s))= |t-s|$, which exists by standard metric space theory, see \cite{BBI01} for example.),
\[
|F(\tilde\gamma(t)) + F(\tilde\gamma(s))| \leq c_Z\rho_\mu(\gamma(t),\gamma(s))\sup_{\tau\in[t,s]}\abs{\nabla F(\tau)} = c_Z|t-s|\sup_{\tau\in[t,s]}\abs{\nabla F(\tau)}
\]
Thus $F\circ\tilde\gamma$ is a Lipschitz function for all $F$ with bounded gradients. Thus by \cite{Fed}, $\gamma$ is almost everywhere differentiable, and 
$
\abs{d\tilde\gamma/dt} < c_Z $ almost everywhere.
\end{proof}

\subsection{Harmonic Coordinates on the Sierpinski Gasket}\label{harmsg}

Let $\SG$ be the Sierpinski gasket with resistance form $\eng$ as in \cite{Kig01,Str06}, and let $\Phi = (\phi_1,\phi_2)$ be harmonic coordinates as in \cite{Kig93-2,Kaj12,Kaj13}. 
In \cite{Kig93-2}, it is shown that for $f,g\in C^1(\R)$ with $f-g\in \idea_{\SG_\Phi}$ if an only if $\nabla f = \nabla g$ on $\SG_\Phi$ . This implies that $T_x^* \SG_\Phi = T_x^*\R^2$ for $x\in \SG_\Phi$.  With this in mind, in this section we refer to $d^{\SG_\Phi}$ as simply $d$. Further, it is proven in \cite{Kig93-2} that $\Phi$ is a coordinate sequence.

We shall consider the space $\ltwoform(\SG_\Phi,\nu)$ where $\nu = \Gamma(\phi^1)+\Gamma(\phi^2)$ is the Kusuoka measure. With the natural norm
\[
\gen{\omega_1dx^1+\omega_2 dx^2,\eta_1dx^1+\eta_2 dx^2} = \int \omega_1\eta_1 + \omega_2\eta_2 \ d\nu 
\]

%Notice that,
%\[
%\norm{d^{SG_\Phi} f}_{\mathcal{H}}=\int \abs{\frac{\partial f}{\partial x^1}}^2 \ d\Gamma(\phi_1)+\abs{\frac{\partial f}{\partial x^2}}^2 \ d\Gamma(\phi_2) + 2\frac{\partial f}{\partial x^1}\frac{\partial f}{\partial x^2}\ d\Gamma(\phi_1,\phi_2)= \eng(f\circ \Phi).
%\]
In \cite{Kig93-2} it is shown that there is a tensor field $Z_x$ such that 
\begin{align*}
\int \gen{d f(x), Z_x d f(x)} \ d\tilde{\nu}(x) 
& = \eng(f\circ \Phi)=\norm{\pi d f }_{\mathcal H} .
\end{align*}
Since
\[
\norm{\pi d f}_{\mathcal{H}}=\int \abs{\frac{\partial f}{\partial x^1}}^2 \ d\Gamma(\phi_1)+\abs{\frac{\partial f}{\partial x^2}}^2 \ d\Gamma(\phi_2) + 2\frac{\partial f}{\partial x^1}\frac{\partial f}{\partial x^2}\ d\Gamma(\phi_1,\phi_2)= \eng(f\circ \Phi),
\]
we see that
\[
(Z_x)_{ij} = \frac{d\Gamma(\phi^i,\phi^j)}{d\nu}.
\]
This also implies that
\begin{align}\label{simpleZ}
\int g^2\gen{d f(x), Z_x d f(x)} \ d\tilde\nu(x) = \norm{d f}_\mathcal{H} = \norm{\phi_1\otimes (g \frac{\partial f}{\partial x^1})+\phi_2\otimes (g \frac{\partial f}{\partial x^2})}_{\mathcal{H}}
\end{align}

For the operator $\pi: \Omega^1_L(\SG_\Phi) \to\mathcal{H} $
\[
\pi (\omega_1 d x^1 +\omega_2 d x^2)= \phi^1\otimes \omega_1+\phi^2\otimes \omega_2.
\]

%----------------------------------------------------------
%Bibilography
%----------------------------------------------------------
\bibliography{Fractals}{}

\newcommand{\etalchar}[1]{$^{#1}$}
\def\cprime{$'$}
\providecommand{\bysame}{\leavevmode\hbox to3em{\hrulefill}\thinspace}
\providecommand{\MR}{\relax\ifhmode\unskip\space\fi MR }
\providecommand{\MRhref}[2]{%
  \href{http://www.ams.org/mathscinet-getitem?mr=#1}{#2}
}
\providecommand{\href}[2]{#2}
\begin{thebibliography}{HKM{\etalchar{+}}16}

\bibitem[BBI01]{BBI01}
Dmitri Burago, Yuri Burago, and Sergei Ivanov, \emph{A course in metric
  geometry}, Graduate Studies in Mathematics, vol.~33, American Mathematical
  Society, Providence, RI, 2001. \MR{1835418 (2002e:53053)}

\bibitem[Bel92]{Bel92}
J.~Bellissard, \emph{Renormalization group analysis and quasicrystals}, Ideas
  and methods in quantum and statistical physics ({O}slo, 1988), Cambridge
  Univ. Press, Cambridge, 1992, pp.~118--148. \MR{1190523 (93k:81045)}

\bibitem[BH91]{BH91}
Nicolas Bouleau and Francis Hirsch, \emph{Dirichlet forms and analysis on
  {W}iener space}, de Gruyter Studies in Mathematics, vol.~14, Walter de
  Gruyter \& Co., Berlin, 1991. \MR{1133391 (93e:60107)}

\bibitem[BSSS12]{BSSS12}
Laurent Bartholdi, Thomas Schick, Nat Smale, and Steve Smale, \emph{Hodge
  theory on metric spaces}, Found. Comput. Math. \textbf{12} (2012), no.~1,
  1--48, Appendix B by Anthony W. Baker. \MR{2886155}

\bibitem[CGIS13]{CGI+13}
Fabio Cipriani, Daniele Guido, Tommaso Isola, and Jean-Luc Sauvageot,
  \emph{Integrals and potentials of differential 1-forms on the {S}ierpinski
  gasket}, Adv. Math. \textbf{239} (2013), 128--163. \MR{3045145}

\bibitem[CS07]{CS07}
Mihai Cucuringu and Robert~S. Strichartz, \emph{Self-similar energy forms on
  the {S}ierpinski gasket with twists}, Potential Anal. \textbf{27} (2007),
  no.~1, 45--60. \MR{2314188 (2008c:31005)}

\bibitem[DABK83]{DABK83}
Eytan Domany, Shlomo Alexander, David Bensimon, and Leo~P. Kadanoff,
  \emph{Solutions to the {S}chr\"odinger equation on some fractal lattices},
  Phys. Rev. B (3) \textbf{28} (1983), no.~6, 3110--3123. \MR{717348
  (85h:82033)}

\bibitem[Fed69]{Fed}
Herbert Federer, \emph{Geometric measure theory}, Die Grundlehren der
  mathematischen Wissenschaften, Band 153, Springer-Verlag New York Inc., New
  York, 1969. \MR{0257325 (41 \#1976)}

\bibitem[FOT11]{FOT11}
Masatoshi Fukushima, Yoichi Oshima, and Masayoshi Takeda, \emph{Dirichlet forms
  and symmetric {M}arkov processes}, extended ed., de Gruyter Studies in
  Mathematics, vol.~19, Walter de Gruyter \& Co., Berlin, 2011. \MR{2778606
  (2011k:60249)}

\bibitem[Gig15]{Gig15}
Nicola Gigli, \emph{On the differential structure of metric measure spaces and
  applications}, Mem. Amer. Math. Soc. \textbf{236} (2015), no.~1113, vi+91.
  \MR{3381131}

\bibitem[GL14]{GL14}
M.~{Gordina} and T.~{Laetsch}, \emph{{Weak Convergence to Brownian Motion on
  Sub-Riemannian Manifolds}}, ArXiv e-prints (2014), arXiv 1403.0142,
  Submitted.

\bibitem[Har99]{Har99}
J.~Harrison, \emph{Flux across nonsmooth boundaries and fractal
  {G}auss/{G}reen/{S}tokes' theorems}, J. Phys. A \textbf{32} (1999), no.~28,
  5317--5327. \MR{1720342 (2001k:58017)}

\bibitem[Hin10]{Hin10}
Masanori Hino, \emph{Energy measures and indices of {D}irichlet forms, with
  applications to derivatives on some fractals}, Proc. Lond. Math. Soc. (3)
  \textbf{100} (2010), no.~1, 269--302. \MR{2578475 (2010k:60272)}

\bibitem[HKM{\etalchar{+}}16]{magnets}
J.~{Hyde}, D.~J. {Kelleher}, J.~{Moeller}, L.~G. {Rogers}, and L.~{Seda},
  \emph{{Magnetic Laplacians of locally exact forms on the Sierpinski Gasket}},
  ArXiv e-prints (2016), arXiv:1604.01340.

\bibitem[HKT12]{HKT12}
M.~Hinz, D.~Kelleher, and A.~Teplyaev, \emph{Measures and {D}irichlet forms
  under the {G}elfand transform}, Zap. Nauchn. Sem. S.-Peterburg. Otdel. Mat.
  Inst. Steklov. (POMI) \textbf{408} (2012), no.~Veroyatnost i Statistika. 18,
  303--322, 329--330. \MR{3032223}

\bibitem[HKT15]{HKT13}
Michael Hinz, Daniel~J. Kelleher, and Alexander Teplyaev, \emph{Metrics and
  spectral triples for {D}irichlet and resistance forms}, J. Noncommut. Geom.
  \textbf{9} (2015), no.~2, 359--390. \MR{3359015}

\bibitem[HR16]{HR16}
Michael Hinz and Luke Rogers, \emph{Magnetic fields on resistance spaces}, J.
  Fractal Geom. \textbf{3} (2016), no.~1, 75--93. \MR{3502019}

\bibitem[HRT13]{HRT}
Michael Hinz, Michael R{\"o}ckner, and Alexander Teplyaev, \emph{Vector
  analysis for {D}irichlet forms and quasilinear {PDE} and {SPDE} on metric
  measure spaces}, Stochastic Process. Appl. \textbf{123} (2013), no.~12,
  4373--4406. \MR{3096357}

\bibitem[HT14]{HT14}
M.~{Hinz} and A.~{Teplyaev}, \emph{{Local Dirichlet forms, Hodge theory, and
  the Navier-Stokes equations on topologically one-dimensional fractals}},
  Trans. Amer. Math. Soc. (2014).

\bibitem[IRT12]{IRT12}
Marius Ionescu, Luke~G. Rogers, and Alexander Teplyaev, \emph{Derivations and
  {D}irichlet forms on fractals}, J. Funct. Anal. \textbf{263} (2012), no.~8,
  2141--2169. \MR{2964679}

\bibitem[Kaj12]{Kaj12}
Naotaka Kajino, \emph{Heat kernel asymptotics for the measurable {R}iemannian
  structure on the {S}ierpinski gasket}, Potential Anal. \textbf{36} (2012),
  no.~1, 67--115. \MR{2886454}

\bibitem[Kaj13]{Kaj13}
\bysame, \emph{Analysis and geometry of the measurable {R}iemannian structure
  on the {S}ierpi\'nski gasket}, Fractal geometry and dynamical systems in pure
  and applied mathematics. {I}. {F}ractals in pure mathematics, Contemp. Math.,
  vol. 600, Amer. Math. Soc., Providence, RI, 2013, pp.~91--133. \MR{3203400}

\bibitem[Kig93]{Kig93-2}
J.~Kigami, \emph{Harmonic metric and {D}irichlet form on the {S}ierpi\'nski
  gasket}, Asymptotic problems in probability theory: stochastic models and
  diffusions on fractals ({S}anda/{K}yoto, 1990), Pitman Res. Notes Math. Ser.,
  vol. 283, Longman Sci. Tech., Harlow, 1993, pp.~201--218. \MR{1354156
  (96m:31014)}

\bibitem[Kig01]{Kig01}
Jun Kigami, \emph{Analysis on fractals}, Cambridge Tracts in Mathematics, vol.
  143, Cambridge University Press, Cambridge, 2001. \MR{1840042 (2002c:28015)}

\bibitem[Kig08]{Kig08}
\bysame, \emph{Measurable {R}iemannian geometry on the {S}ierpinski gasket: the
  {K}usuoka measure and the {G}aussian heat kernel estimate}, Math. Ann.
  \textbf{340} (2008), no.~4, 781--804. \MR{2372738 (2009g:60105)}

\bibitem[Sma15]{Sma15}
Nat Smale, \emph{A {H}odge theory for {A}lexandrov spaces with curvature
  bounded from above}, Anal. Appl. (Singap.) \textbf{13} (2015), no.~3,
  291--301. \MR{3318963}

\bibitem[SS12]{SS12}
Nat Smale and Steve Smale, \emph{Abstract and classical {H}odge--de {R}ham
  theory}, Anal. Appl. (Singap.) \textbf{10} (2012), no.~1, 91--111.
  \MR{2876937}

\bibitem[Sto10]{Sto10}
Peter Stollmann, \emph{A dual characterization of length spaces with
  application to {D}irichlet metric spaces}, Studia Math. \textbf{198} (2010),
  no.~3, 221--233. \MR{2650987 (2011i:30052)}

\bibitem[Str06]{Str06}
Robert~S. Strichartz, \emph{Differential equations on fractals}, Princeton
  University Press, Princeton, NJ, 2006, A tutorial. \MR{2246975 (2007f:35003)}

\bibitem[Stu94]{Stu94}
Karl-Theodor Sturm, \emph{Analysis on local {D}irichlet spaces. {I}.
  {R}ecurrence, conservativeness and {$L^p$}-{L}iouville properties}, J. Reine
  Angew. Math. \textbf{456} (1994), 173--196. \MR{1301456 (95i:31003)}

\end{thebibliography}
\bibliographystyle{amsalpha}
%
%End
%---------------------------------------------------------
\end{document}